\documentclass[final]{siamltex}

\usepackage{graphicx}
\usepackage{amssymb,amsmath}
\usepackage{enumitem}
\usepackage{graphicx}
\usepackage{psfrag}

\usepackage[dvips]{hyperref}
\usepackage{breakurl}

\newtheorem{remark}[theorem]{Remark}

\newcommand{\R}{\mathbb R}

\newcommand{\wave}{\mathcal W}
\newcommand{\M}{\mathcal   M}
\newcommand{\radon}{\mathcal R}
\newcommand{\hilbert}{\mathcal  H}
\newcommand{\K}{\mathcal K}
\newcommand{\ubp}{\mathcal{B}_\Om}

\newcommand{\rmd}{\mathrm d}
\newcommand{\coloneqq}{:=}

\newcommand\abs[1]{\left\vert#1\right\vert}
\newcommand\sabs[1]{\lvert#1\rvert}

\newcommand\set[1]{\left\{#1\right\}}

\newcommand\sset[1]{\{#1\}}
\newcommand\inner[1]{\left\langle#1\right\rangle}
\newcommand\edot{\,\cdot\,}

\newcommand{\Om}{\Omega}
\newcommand{\om}{\omega}

\newcommand{\DD}{\mathcal{D}}

\newcommand{\eps}{\epsilon}

\newcommand{\kl}[1]{\left(#1\right)}
\newcommand{\mkl}[1]{\bigl(#1\bigr)}
\newcommand{\skl}[1]{(#1)}
\newcommand{\req}[1]{(\ref{eq:#1})}

\title{Universal Inversion Formulas for\\Recovering a Function from Spherical Means}

\author{Markus Haltmeier\thanks{Department of Mathematics, University of Innsbruck,
Technikestra{\ss}e 21a, A-6020 Innsbruck, Austria
(\href{mailto:markus.haltmeier@uibk.ac.at}{\tt markus.haltmeier@uibk.ac.at}).}}

\begin{document}
\maketitle

\begin{abstract}
The problem of reconstruction a function from  spherical means is at the
heart of several modern imaging modalities and  other  applications.
In this paper we derive universal  back-projection type reconstruction formulas
for recovering a function in arbitrary dimension from averages over spheres
centered on the boundary of an arbitrarily shaped bounded convex domain with smooth boundary.
Provided  that the unknown  function  is supported inside that domain, the derived  formulas recover  the
unknown function up to an  explicitly computed integral operator.
For elliptical domains the integral operator is shown to vanish and hence  we establish  exact inversion
formulas for recovering a function from spherical means centered on the boundary of
elliptical domains in arbitrary dimension.
\end{abstract}



\pagestyle{myheadings}
\thispagestyle{plain}
\markboth{MARKUS HALTMEIER}{UNIVERSAL  INVERSION  FORMULAS FOR
SPHERICAL MEANS}

\section{Introduction}

Let $\Om\subset \R^n$ be a bounded convex domain  in $\R^n$
with smooth boundary. In this paper we study the problem of recovering a function
$f \colon \R^n \to \R$  that is supported in $\Om$ from the
averages (spherical means)
\begin{equation} \label{eq:sm}
\kl{\M f} \kl{x, r}
\coloneqq
\frac{1}{\omega_{n-1}}
\int_{S^{n-1}} f \kl{ x + r \sigma} \rmd S\kl{\sigma}
\end{equation}
over spherical surfaces
with  centers $ x \in \partial \Om$ and radii $r>0$.
Here   $S^{n-1} \subset \R^n$ is the $n-1$ dimensional
unit sphere, $\om_{n-1}$ its total surface area and $\rmd S$
denotes the standard surface measure. Further recall that a domain is an open, connected, nonempty set. We develop closed form reconstruction formulas of the backprojection type
for recovering the function $f$ from its spherical means $\M f \kl{x, r}$
defined by \req{sm}.
The derived  formulas can be applied to
arbitrarily shaped domains in arbitrary dimensions and recover
the unknown function modulo an explicitly computed integral operator
$\K_\Om$.  For elliptical domains, the operator $\K_\Om$ is shown to vanish.
We therefore establish  exact  reconstruction formulas of the
backprojection type in these cases.
Our results  generalize the ones recently obtained in  \cite{Nat12} for $n =3$
and in \cite{Hal13a} for $n =2$ to the case of
arbitrary spatial dimension.

\begin{psfrags}
\psfrag{O}{$\Omega$}
\psfrag{S}{ $\partial \Omega$}
\psfrag{f}{$f$}
\psfrag{B}{$\set{ x_1 \in \R^n :  \abs{x_1-x} =r}$}
\psfrag{x}{$x$}
\begin{figure}
\centering
\includegraphics[width=0.7\textwidth]{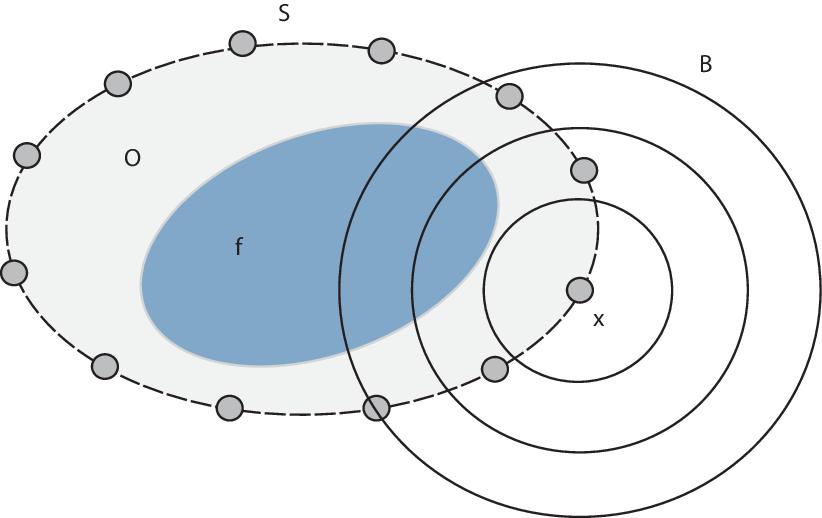}
\caption{\emph{Recovering a function from spherical means.}
Suppose  that the function $f$ (representing some physical quantity of
interest) is supported inside  the domain  $\Om$. Detectors are placed at various
locations $x$ on the boundary $\partial \Om$ of the domain and record
averages of $f$  over  spherical surfaces $\set{ x_1  \in \R^n : \abs{x_1-x} =r}$ with  radii $r>0$. In this paper we derive explicit formulas for recovering the function
$f$ from these spherical averages (see Theorems~\ref{thm:even},~\ref{thm:odd} and~\ref{thm:ell}).
\label{fig:setting}}
\end{figure}
\end{psfrags}

The problem of recovering a function from spherical means  is at the heart of many modern  imaging applications,
where the centers of the spheres of integration  correspond to admissible locations
of detectors recording some physical quantity encoded in
$f$; see Figure~\ref{fig:setting}.
For example, recovering a function from spherical means is essential
for the hybrid imaging techniques photoacoustic tomography (PAT) and
thermoacoustic tomography (TAT) where the function $f$ models the initial pressure of
the acoustic field induced by a short electromagnetic pulse.
In these  applications the inversion from spherical means  arises in
three spatial dimensions (see~\cite{FinRak09,KucKun08, XuWan06}) as well as in two spatial dimensions in variants of PAT/TAT using
integrating detectors (see \cite{BurBauGruHalPal07,PalNusHalBur07a,ZanSchHal09b})
instead of the more common point like detectors.
In fact  these applications initiated the authors interest in  the
problem of recovering a function from spherical means.
The inversion from spherical means is, however, is also  essential for other technologies such as SONAR (see \cite{BelFel09,QuiRieSch11}),
SAR imaging (see \cite{And88,RedPay03}),
ultrasound tomography (see \cite{Nor80,NorLin81}),
or seismic imaging (see \cite{BleCohSto01,Faw85}).

\subsection{Main results}

Before presenting our  main results, we  introduce some notation. For any integrable function $\varphi \colon \R^n \to \R$, we define
the Radon transform
\begin{equation*}
\kl{\radon \varphi} \kl{\omega,s}
:=
\int_{\omega^\bot}  \varphi\kl{ s \omega  +  y}
\rmd S\kl{y}
\qquad
\text{ for  }  \; \kl{\omega, s} \in S^{n-1} \times \R \,,
\end{equation*}
where $\omega^\bot \coloneqq \set{y \in \R^n: \omega \cdot y = 0}$ denotes
the hyperplane consisting of all vectors orthogonal to $\omega \in S^{n-1}$.
The derivative of a function  $\psi \colon S^{n-1} \times \R \to \R$
in the second argument will be denoted by
$\kl{\partial_s \psi} \kl{\omega, s}$, and   $\kl{\hilbert_s  \psi} \kl{\omega, s} $
will be used to denote the Hilbert transform  in the second argument
(defined as the  convolution with the principal value distribution
$1/\skl{\pi s}$).
Further, for two distinct points  $x_0, x_1 \in \R^{n}$, we set
\begin{equation} \label{eq:nr}
	\omega_\star\kl{x_0,x_1}
	\coloneqq
	\frac{x_1-x_0}{\abs{x_1-x_0}}
	\,,
	\quad
	s_\star\kl{x_0,x_1}
	\coloneqq
	\frac{1}{2} \; \frac{\abs{x_1}^2 - \abs{x_0}^2}{\abs{x_1-x_0}}
	\,.
\end{equation}
As illustrated  in Figure~\ref{fig:geometry}, the set  $H_\star\kl{x_0,x_1} =
\sset{ x \in \R^n : \omega_\star\kl{x_0,x_1} \cdot x   = s_\star\kl{x_0,x_1}}$ is the hyperplane
of all points having the same distance between $x_0$ and $x_1$. The unit vector
$\omega_\star\kl{x_0,x_1}$ is orthogonal to the plane $H_\star\kl{x_0,x_1}$ and $s_\star\kl{x_0,x_1}$ is the oriented distance of that plane from the origin.

\begin{psfrags}
\psfrag{a}{\scriptsize $x_0$}
\psfrag{b}{\scriptsize $x_1$}
\psfrag{s}{\scriptsize $s_\star\kl{x_0,x_1}$}
\psfrag{O}{\scriptsize domain $\Omega$}
\psfrag{S}{\scriptsize boundary  $\partial \Omega$}
\psfrag{m}{\scriptsize$\kl{x_1+x_0}/2$}
\psfrag{E}{\scriptsize$H_\star\kl{x_0,x_1}$}
\begin{figure}
\centering
\includegraphics[width=0.6\textwidth]{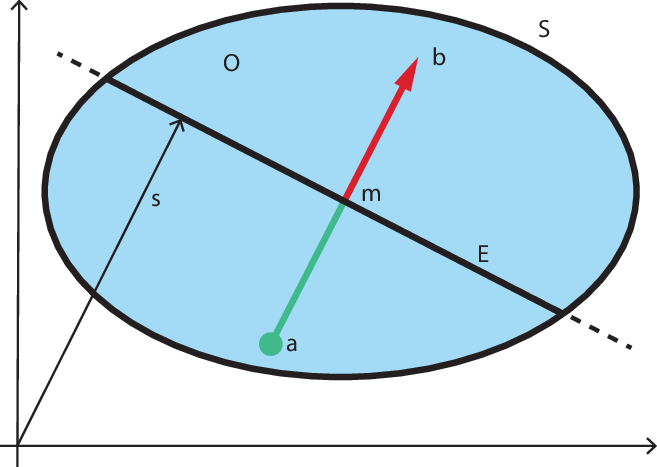}
\caption{\emph{Mid-plane between  two distinct points $x_0$ and $x_1$ in  $\R^n$.}
The hyperplane $H_\star\kl{x_0,x_1} = \sset{ x \in \R^n : \omega_\star\kl{x_0,x_1} \cdot x   = s_\star\kl{x_0,x_1}}$
is the mid-plane between $x_1$ and $x_0$, that is,  consists of all points having the same distance between
these two points.
The unit vector $\omega_\star\kl{x_0,x_1}    =  \skl{x_1-x_0}/ \sabs{x_1-x_0}$  is orthogonal to the plane
$H_\star\kl{x_0,x_1}$  and $s_\star\kl{x_0,x_1} = \kl{\sabs{x_1}^2 - \sabs{x_0}^2} / \kl{2 \sabs{x_1 -x_0}}$
is the oriented distance of that plane from the origin.\label{fig:geometry}}
\end{figure}
\end{psfrags}

\subsubsection*{Inversion on general domains}

Suppose that $\Om \subset \R^{n}$ is a bounded convex domain in
$\R^n$ with smooth boundary $\partial \Om$ and denote by
$\chi_\Om \colon \R^n \to \R$ the characteristic function of
$\Om$ (taking the value one inside the domain $\Om$ and the value zero outside).   Further, for  any  $C^\infty$ function $f \colon \R^n \to \R$  that is supported  inside $\Om$ and every $x_0 \in \Om$, we define
\begin{equation} \label{eq:K}
\kl{\K_\Om f} \kl{x_0} \coloneqq
\int_{\Om} k_\Om \kl{x_0,x_1} f\kl{x_1}  \rmd x_1 \,,
\end{equation}
with
\begin{equation}\label{eq:kern}
	 k_\Om \kl{x_0,x_1}
	  \coloneqq
	   \begin{cases} 
	  \frac{\kl{-1}^{\kl{n-2}/2}}{2^{n+1} \pi^{n-1}}
	  \,
	  \frac{\kl{ \partial_s^n \hilbert_s  \radon \chi_\Om}
	  \kl{\omega_\star\kl{x_0,x_1},  s_\star\kl{x_0,x_1}}}
	  {  \abs{x_1-x_0}^{n-1}}	
	& \text{ if  $n$ is even}
	\\[1.4em]
	 \frac{\kl{-1}^{\kl{n-1}/2}}{2^{n+1}\pi^{n-1}}
	 \,
	 \frac{\kl{ \partial_s^n  \radon \chi_\Om}
	\kl{\omega_\star\kl{x_0,x_1},  s_\star\kl{x_0,x_1}}}
	{\abs{x_1-x_0}^{n-1}}
	& \text{ if  $n$ is odd}
	\end{cases} \,.
\end{equation}
Here and in similar situations, $\partial_s^n$ denotes
the $n$-fold composition of the differentiation operator
$\partial_s$. Note that  $\omega_\star\kl{x_0,x_1}$,  $s_\star\kl{x_0,x_1}$ and $k_\Omega\kl{x_0,x_1}$
are only defined  when $x_0 \neq  x_1$.

\begin{remark}\label{rem:KOm}
Since  $\Omega$ is assumed to be a  convex domain
with $C^\infty$ boundary, the Radon transform of $\chi_\Omega$ is a smooth function except for pairs $\kl{\om, s} \in S^{n-1} \times \R$  where the corresponding
plane  $ \sset{ x \in \R^n : \omega \cdot x   = s}$ is tangential to
the boundary $\partial \Omega$. For two distinct points $x_0, x_1$ inside the domain $\Om$,
the mid-plane $ H_\star \kl{x_0, x_1} = \sset{ x \in \R^n :\omega_\star\kl{x_0,x_1} \cdot x   = s_\star\kl{x_0,x_1}}$ between these points
is never tangential to the boundary of the domain (see Figure~\ref{fig:geometry})
and further the operators $\partial_s^n$ and $\hilbert_s$
preserve the locations of singularities. Consequently, for any compact subset $K \subset \Omega$, the
functions $\skl{ \partial_s^n \hilbert_s  \radon \chi_\Om}
\skl{\omega_\star\kl{x_0,x_1},  s_\star\kl{x_0,x_1}}$
and $\skl{ \partial_s^n  \radon \chi_\Om}
\skl{\omega_\star\kl{x_0,x_1},  s_\star\kl{x_0,x_1}}$  are  $C^\infty$ and
bounded  on $\set{\kl{x_0,x_1} \in  K \times K \colon x_0\neq x_1}$.
This shows, that $k_\Om$ is  a weakly singular kernel on $K \times K$
and that the  integral operator
$\K_\Om \colon L^2\kl{K} \to L^2\kl{K}$ is well defined and compact.
 \end{remark}

The inversion formulas we establish in this paper
are exact modulo the integral operator $\K_\Om$.
They look somewhat different in even and in odd
dimensions and are stated in separate theorems
below.

\smallskip
In even dimensions, our main result is as follows.

\begin{theorem}[Inversion in even dimension] \label{thm:even}
Let $n \geq 2$ be an even natural number, let
$\Om \subset \R^n $ be a bounded convex domain with smooth
boundary $\partial \Om$, and  let $f \colon \R^n \to \R$ be $C^\infty$ and  supported inside $\Om$.

Then, for every $x_0 \in \Om$,
\begin{align}\nonumber
	f\kl{x_0}
	&=
	   \kl{\K_\Om  f}\kl{x_0}
	   +
	   \frac{\skl{-1}^{\skl{n-2}/2}
\om_{n-1}}{2 \pi^{n}} \, \times
	 \\ \label{eq:even-a}
           &
            \hspace{0.15\textwidth}  \nabla_{x_0} \cdot
	 \int_{\partial \Om}
	 \nu_x
	 \int_{0}^\infty
	 \frac{\kl{ r \DD_r^{n-2}
	 r^{n-2}\M f} \kl{x, r} }
	 {r^2 - \abs{x_0-x}^2}
	 \; \rmd r\rmd S\kl{x} \,,
\\[0.4em] \nonumber
	f\kl{x_0}
	&=
	   \kl{\K_\Om  f}\kl{x_0}
	   +
	  \frac{\skl{-1}^{\skl{n-2}/2}
	  \om_{n-1}}{2 \pi^{n}} \, \times
	  \\ \label{eq:even-b}
	& \hspace{0.15\textwidth}  \int_{\partial \Om} \nu_x \cdot  \kl{x_0-x}\int_{0}^\infty
	\frac{ \kl{\partial_{r} \DD_r^{n-2}
	 r^{n-2}\M f}  \kl{x, r}}
	{r^2 - \abs{x_0-x}^2}
	\; \rmd r
	\rmd S\kl{x}
	\,.
\end{align}
In both formulas, the inner integration is taken in the principal
value sense, $\rmd S$ is the usual  surface measure,
$\nu_x$ is the outward pointing  unit normal to $\Om$,
and $\K_\Om$ is  the integral operator defined by \req{K}, \req{kern}.
Moreover, $\DD_r \coloneqq  \skl{2r}^{-1}\partial_r$ denotes differentiation with respect to $r^2$
and  $\nabla_{x_0} \cdot$  the divergence with respect to $x_0$.
\end{theorem}

\begin{proof}
See Section~\ref{sec:even}.
\end{proof}

\medskip
In  odd dimension we have the following
corresponding result.

\begin{theorem}[Inversion in odd dimension] \label{thm:odd}
Let $n \geq 3$ be an odd natural number, let   $\Om \subset \R^n $ be a bounded
convex domain with smooth boundary $\partial \Om$,
and let $f \colon \R^n \to \R$ be $C^\infty$ and supported inside $\Om$.

Then, for every $x_0 \in \Om$,
\begin{align}\nonumber
	f\kl{x_0}
	&=
	   \kl{\K_\Om  f}\kl{x_0}
	   +
	   \frac{\skl{-1}^{\skl{n-3}/2}
	   \om_{n-1}}{4 \pi^{n-1}} \, \times
	 \\ \label{eq:odd-a}
           &
           \hspace{0.14\textwidth} \nabla_{x_0} \cdot
	 \int_{\partial \Om}
	 \nu_x
	 \kl{ \DD_r^{n-2} r^{n-2} \M f} \kl{x, \abs{x_0-x}}
	 \rmd S\kl{x} \,,
\\[0.4em] \nonumber
	f\kl{x_0}
	&=
	   \kl{\K_\Om  f}\kl{x_0}
	   +
	   \frac{\skl{-1}^{\skl{n-3}/2}
	   \om_{n-1}}{4\pi^{n-1}} \, \times
	  \\ \label{eq:odd-b}
	& \hspace{0.124\textwidth}
	\int_{\partial \Om} \nu_x \cdot  \frac{x_0-x}{\sabs{x_0-x}}
	\kl{ \partial_{r} \DD_r^{n-2}
	r^{n-2} \M f} \kl{x, \abs{x_0-x}}
	\rmd S\kl{x}
	\,.
\end{align}
Here $\K_\Om$, $\nu_x$, $\nabla_{x_0}$, $\rmd S$, and $\DD_r$ are as in
Theorem~\ref{thm:even}.
\end{theorem}

\begin{proof}
See Section~\ref{sec:odd}.
\end{proof}

Both, Theorem \ref{thm:even} and Theorem  \ref{thm:odd} will
follow from corresponding  statements  for the inversion of the wave
equation in even and odd dimensions,
which we shall establish in the following sections
(see Theorems~\ref{thm:wave-even} and \ref{thm:wave-odd}).

\subsubsection*{Exact reconstruction for elliptical domains}

In the case that $\Om$ is an elliptical domain, we show that the integral operator
$\K_\Om$  vanishes exactly and therefore Theorems
\ref{thm:even} and \ref{thm:odd} provide exact reconstruction
formulas for ellipsoids. After translation and rotation, we may assume that the
elliptical domain takes the standard form
\begin{equation*} 
	\Om
	\coloneqq  \set{ x \in \R^n :
	\abs{ A^{-1} x }^2  < 1 } \,,
\end{equation*}
where $A \in \R^{n \times n}$ is a diagonal matrix with positive
(possibly distinct) diagonal entries. Obviously, a ball  is a special case of an
elliptical domain where all diagonal entries of $A$ coincide and are equal to the
radius of the ball.

\medskip
For elliptical domains we have the following exact inversion
formulas.

\begin{theorem}[Exact inversion on elliptical domains] \label{thm:ell}
Let $\Om \subset\R^n$ be an elliptical domain.
Then  $\K_{\Om} f$ vanishes identically on $\Om$.
In particular, for any smooth  function $f\colon \R^n \to \R$
that is supported inside $\Om$ and every $x_0 \in \Om$, the following hold:
\begin{enumerate}[itemsep=0em,topsep=0em,label=(\alph*)]
\item \label{it:ell-even}
If $n$ is even, then
\begin{align}\nonumber
	f\kl{x_0}
	&=\frac{\skl{-1}^{\skl{n-2}/2}
	   \om_{n-1}}{2 \pi^{n}}\, \times
	   \\ \label{eq:inv-ell-even-a}
	   & \hspace{0.05\textwidth}
	   \nabla_{x_0} \cdot \int_{\partial \Om}
	   \nu_x
	   \int_{0}^\infty
	   \frac{\kl{r \DD_r^{n-2}
	   r^{n-2}\M f} \kl{x, r} }
	  {r^2 - \abs{x_0-x}^2}
	  \, \rmd r\rmd S\kl{x} \,,
\\[0.4em] \nonumber
	f\kl{x_0}
	&=
	  \frac{\skl{-1}^{\skl{n-2}/2}
	  \om_{n-1}}{2 \pi^{n}} \, \times
	  \\ \label{eq:inv-ell-even-b}
	& \hspace{0.05\textwidth}  \int_{\partial \Om} \nu_x \cdot  \kl{x_0-x}\int_{0}^\infty
	\frac{ \kl{\partial_{r} \DD_r^{n-2}
	 r^{n-2}\M f}  \kl{x, r}}
	{r^2 - \abs{x_0-x}^2}
	\; \rmd r
	\rmd S\kl{x}
	\,.
\end{align}

\item \label{it:ell-odd}
If $n$ is odd, then
\begin{align} \nonumber
	f\kl{x_0}
	&=
	   \frac{\skl{-1}^{\skl{n-3}/2}
	   \om_{n-1}}{4\pi^{n-1}} \, \times
	   \\ \label{eq:inv-ell-odd-a}
	   & \hspace{0.06\textwidth}
	   \nabla_{x_0} \cdot
	    \int_{\partial \Om}
	 \nu_x
	 \kl{ \DD_r^{n-2} r^{n-2} \M f} \kl{x, \abs{x_0-x}}
	 \rmd S\kl{x}
	 \,,
	 \\[0.4em] \nonumber
	f\kl{x_0}
	&=
	   \frac{\skl{-1}^{\skl{n-3}/2}
	   \om_{n-1}}{4\pi^{n-1}} \, \times
	  \\ \label{eq:inv-ell-odd-b}
	& \hspace{0.06\textwidth}
	\int_{\partial \Om} \nu_x \cdot \frac{x_0-x}{\sabs{x_0-x}}
	\kl{ \partial_{r} \DD_r^{n-2}
	r^{n-2} \M f} \kl{x, \abs{x_0-x}}
	\rmd S\kl{x} \,.
\end{align}
\end{enumerate}
Here  $\nu_x$, $\nabla_{x_0}$, $\rmd S$, and $\DD_r$ are as in
Theorem~\ref{thm:even}.
\end{theorem}

\begin{proof}
See Section \ref{sec:ell}.
\end{proof}

By taking limits one can easily establish exact inversion formulas like \req{inv-ell-even-a}--\req{inv-ell-odd-b} for certain unbounded domains, such as for elliptical cylinders. We omit  formulating  such generalizations.
Further, it  would be interesting so find further domains
$\Omega$, where the integral operator $\K_\Om$ can be shown to vanish. Such an investigation, however, is beyond the scope of this paper.

\subsection{Relations to previous  work}

Exact back-projection type inversion formulas for recovering a function from
spherical means  with centers on the boundary of  a ball have been discovered
quite recently in \cite{FinHalRak07,FinPatRak04, Kun07a,Ngu09,XuWan05}.
In \cite{XuWan05} a formula has been found for $n = 3$,
which has later been generalized to arbitrary dimensions in \cite{Kun07a}.
In odd dimension these  formulas coincides with our
Equation~\req{inv-ell-odd-a} (which, however, holds for the more general
case of elliptical center sets). A different set of exact reconstruction formulas
has been derived in \cite{FinPatRak04} for odd dimensions and in
\cite{FinHalRak07} for even dimensions. In~\cite{FinRak09} relations between
the different formulas have been investigated for dimensions $n=2$ and
$n =3$. None of these papers considers the case of more general
domains. In~\cite{Kun11} reconstruction formulas of the back-projection type have been
found for certain polygons and polyhedra in two and three  spatial dimensions.

Formulas that recover a function from spherical means with centers
on the boundary of an elliptical domain in arbitrary dimension have been obtained
in~\cite[Equations~(20),~(21)]{Pal12}. The derived identities
as well as the method of proof are different from ours.
Our results are, however, closely related to ones of~\cite{Hal13a,Nat12}.
Actually, the present article generalizes the result obtained
for $n=2$ in \cite{Hal13a}
and for $n=3$ in \cite{Nat12}
to the case of arbitrary spatial dimension.

\subsection{Outline}

The main aim of the following sections is the proof of
Theorems~\ref{thm:even},~\ref{thm:odd} and~\ref{thm:ell}.
To that end, we first derive an   auxiliary identity for the wave equation in Section~\ref{sec:wave-aux}
(see Theorem~\ref{thm:ubp}).
Subsequently, in Section~\ref{sec:even} we shall prove
Theorem~\ref{thm:even} and in Section~\ref{sec:odd} we
establish Theorem~\ref{thm:odd}.
In these sections, we also derive corresponding statement for recovering the initial data of the wave equation
from the solution  on the boundary of  an arbitrarily shaped domain. These results, which  are also
of  interest in their own, will be presented in
Theorems~\ref{thm:wave-even} and~\ref{thm:wave-odd} below.
In Section~\ref{sec:ell} we consider the case of elliptical domains, where we show that the operator
$\K_\Om$ vanishes identically and therefore we establish the exact reconstruction
formulas stated in Theorem~\ref{thm:ell}.
The paper concludes with a  discussion in Section~\ref{sec:discussion}.

\section{Auxiliary  results for the wave equation}
\label{sec:wave-aux}

Let  $\Om \subset \R^n$ be a bounded convex domain in
$\R^n$ with smooth boundary $\partial \Om$ and  let
$f \colon  \R^n\to \R $ be a smooth function that is supported inside $\Om$.
Consider the following initial value problem for the wave equation
\begin{equation}  \label{eq:wave}
	\left\{ \begin{aligned}
	\kl{\partial_t ^2  - \Delta_{x} } p\kl{x,
	t}
	&=
	0 \,,
	 & \text{ for }
	\kl{x,t} \in
	\R^n \times \kl{0, \infty}
	\\
	p\kl{x,0}
	&=
	f\kl{x} \,,
	& \text{ for }
	x  \in \R^n
	\\
	\partial_t
	p\kl{x,0}
	&=0 \,,
	& \text{ for }
	x  \in \R^n
\end{aligned} \right.\,.
\end{equation}
Here $\partial_t$ denotes differentiation with respect to the temporal variable
 $t \in \kl{0, \infty}$ and  $\Delta_{x}$ is the  Laplacian in the spatial variable $x \in \R^n$.  To indicate the dependance
 of the initial data
 we will also write   the solution of \req{wave} as $p = \wave f$.

According to the well known explicit formulas for the solution of
\req{wave} in terms of spherical means,  recovering a function from
spherical means is essentially equivalent  to the problem of recovering
the initial data in \req{wave} from values of the solution on $\partial \Om \times \kl{0, \infty}$.  In this section we derive a basic result for the wave inversion
which in the following sections will be applied to derive the results for the
inversion form spherical means presented in the introduction.

\subsection{Outgoing fundamental solution}

Throughout the following  we denote by $G\kl{x, t}$ the outgoing fundamental solution (or free space Green's function)
of the wave equation,  that vanishes on $\set{t <0}$ and satisfies the equation
\begin{equation*}
	\kl{\partial_t^2  - \Delta_{x} }  G\kl{x, t}
 	=   \delta_{n} \kl{x} \delta \kl{t} \,,
	\qquad  \text{ for  all }
	\kl{x, t}
	\in \R^n
	\times  \R \,.
\end{equation*}
Here and in the following, $\delta_n$ and $\delta$ denote the
$n$-dimensional  and one-dimensional delta distribution, respectively.

\begin{remark}
By definition, the outgoing fundamental solution is a
distribution on $\R^n \times \R$.
The  arguments in $G\kl{x, t}$ do not mean a point-evaluation at $\kl{x, t}$
but are only a formal notation indicating the  variables, where this distributions acts on.
Derivatives of the fundamental solution,
like $\partial_t G\kl{x, t}$, will  always denote distributional derivatives.
Further, notice that the mapping $t \mapsto G\kl{\edot, t}$  from $(0, \infty)$ to the space $\mathcal D'\kl{\R^n}$ of distributions on
$\R^n$ is well defined and $C^\infty$.
\end{remark}

With the outgoing fundamental solution of the wave equation,
the solution of the initial  value problem \req{wave} can be
written as
\begin{equation} \label{eq:sol-green}
 	p \kl{x, t}
	=
	\int_{\Om}
	\partial_t G\kl{x- x_1, t}
	f \kl{x_1} \rmd x_1
	\,,
	\qquad \text{ for all } \kl{x, t } \in
	\R^{n} \times \kl{0, \infty} \,.
 \end{equation}
Throughout this paper  integrals  like the one on  the right
hand side in~\req{sol-green} will  always be  understood in the weak sense
(for any fixed $t$).
Hence, identity \req{sol-green} actually means, that
\begin{equation} \label{eq:sol-green-full}
 	\int_{\R^n}
	p \kl{x, t}
	\varphi \kl{x}
    \rmd x
	=
	\int_{\Om}
 	\inner{ \partial_t  G \kl{\edot - x_1,t} ,  \varphi }
	f \kl{x_1}
	\rmd x_1
 \end{equation}
 for any  smooth test function $\varphi  \colon \R^n  \to \R$ with  $\inner{\edot,\edot}$
 denoting the duality  pairing between a distribution and a test function
 on $\R^n$.

After inserting the known explicit  expressions for the outgoing
fundamental solution $G$,  Equation \req{sol-green} can be rewritten in
terms of spherical means of the function $f$.
The explicit solution formulas differ in even and in odd dimensions and will be stated in
Sections \ref{sec:even} and \ref{sec:odd},
where we study the even and the odd dimensional case separately
and in more detail.

\subsection{Kirchhoff integral representation}

The following  Kirchhoff integral representation  relates the initial conditions
 of the free space wave equation \req{wave} with boundary values on some domain.
It is well known for three spatial
dimension (and follows, for example, from \cite[Equation~(4.1.25)]{Fri75}) but we did not found a reference for the  case of arbitrary dimension.
Since the Kirchhoff integral representation serves as the basis of our further  computations, we decided to include a simple derivation based on Greens second identity.

\begin{lemma}[Kirchhoff integral representation]\label{lem:kirch}
Let  $\Om \subset \R^n$ be a bounded domain
with smooth boundary $\partial \Omega$, let $p = \wave f$ denote the solution of \req{wave} with initial data $f \in C^\infty_c \kl{\Omega}$, and let $G$ denote the outgoing fundamental solution of the wave equation.   Then, for every $x_0 \in \Om$, we have
\begin{multline}\label{eq:kirch}
f \kl{x_0}
=
\int_{\partial \Om}
\nu_x \cdot  \int_{\R}
G\kl{x_0-x, t} \nabla_{x} p \kl{x,t}
\rmd t\rmd S\kl{x}
\\
-
\int_{\partial \Om}
\nu_x \cdot
\int_{\R}
\nabla_{x} G\kl{x_0-x, t}
p \kl{x,t}
\rmd t
\rmd S\kl{x} \,.
\end{multline}
\end{lemma}

\begin{proof}
Greens second identity applied with
$G \kl{x_0 - \edot, t}$ and
$p \kl{\edot, t}$ for fixed $\kl{t,x_0} \in (0, \infty) \times \Om$
yields
\begin{multline*}
\int_{\partial \Om}
\nu_x \cdot
\kl{
G\kl{x_0-x, t} \nabla_x p\kl{x,t}
 -
p \kl{x,t} \nabla_x G\kl{x_0-x, t}
 }
\,
\rmd S\kl{x}
\\
\begin{aligned}
&=
\int_{\Om}
\kl{ G\kl{x_0-x, t}
\Delta_x
p\kl{x,t}
-
p\kl{x,t}
\Delta_x G\kl{x_0-x, t}
}
\rmd x
\\
&=
\int_{\Om}
\kl{
G\kl{x_0-x, t}
\partial_t^2
p\kl{x,t}
-
p\kl{x,t}
\partial_t^2 G\kl{x_0-x, t}
}
\rmd x  \,.
\end{aligned}
\end{multline*}
For the second equality we used the assumption, that
$G \kl{x_0 - \edot, t}$ and $p \kl{\edot, t}$ both satisfy the wave equation on $\set{t >0}$.
Integrating the above identity  over some finite time interval
$[T_1, T_2] \subset (0, \infty)$, interchanging the order of integration,
and performing two integration by parts gives
\begin{multline*}
\int_{T_1}^{T_2}
\int_{\partial \Om}
\nu_x \cdot
\kl{
G\kl{x_0-x, t} \nabla_x p\kl{x,t}
 -
p \kl{x,t} \nabla_x G\kl{x_0-x, t}
 }
\,
\rmd S\kl{x}\rmd t
\\
\begin{aligned}
&=
\int_{\Om}
G\kl{x_0-x, T_2}
\partial_t p\kl{x,T_2}
\rmd x
-
\int_{\Om}
 p\kl{x,T_2}
\partial_t G \kl{x_0-x, T_2}
\rmd x
\\
&
-
\int_{\Om}
G\kl{x_0-x, T_1}
\partial_t p\kl{x,T_1}
\rmd x
+
\int_{\Om}
 p\kl{x,T_1}
\partial_t G \kl{x_0-x, T_1}
\rmd x
 \,.
\end{aligned}\end{multline*}
Since any solution of the wave equation with compact support tends to zero
as $t \to \infty$ (uniformly on every bounded set), the first two terms term on the right hand  vanish as
$T_2 \to \infty$. Next, we note that by Duhamel's principle, we have
$ \lim_{t \to 0} G\kl{x_0-\edot, t} =  0$ and
$ \lim_{t \to 0} \partial_t G\kl{x_0-\edot, t} =  \delta \kl{x_0-\edot}$.
Hence the latter difference converges to
$f \kl{x_0}$ as $T_1 \to 0$,  which yields the claimed representation \req{kirch}.
\end{proof}

In the proof of Lemma~\ref{lem:kirch} as well as in the following
derivations we  formally operate with distributions as they were classical functions.
These computations can be made more rigorous by writing down all equalities in
the weak sense (similar as done in~\cite{Hal13a} for the two dimensional case)
and then using classical integral calculus.  However, the calculus using distributions used in the present  paper seems to be more intuitive and easier to follow and has also been
used in~\cite{FinPatRak04,Nat12} for the  derivation of  inversion  formulas
in three spatial dimensions.

\subsection{Universal backprojection}

For any smooth   function $v \colon \partial \Om \times \kl{0,\infty} \to \R$
we define
\begin{equation} \label{eq:ubp}
\kl{\ubp v}\kl{x_0}
\coloneqq
2\, \nabla_{x_0} \cdot \int_{\partial \Om} \nu_x \int_{\R}
G \kl{x- x_0, t} v \kl{x, t } \rmd t \rmd S \kl{x} \,,
\quad \text{ for } x_0 \in \Om \,.
\end{equation}
Note that in  the even dimensional case, the function
$v \kl{x, t}$ needs some decay as $t \to \infty$ in order that the
integral $\kl{\ubp v}\kl{x_0}$ is  well defined.
This  is  certainly the case if $v$ is the restriction of the solution of the
initial value problem \req{wave} for some compactly supported initial data $f$,
which happens in all instances where we apply the operator $\ubp$.
In three  spatial dimensions (and in a slightly different form),
the  inversion integral \req{ubp} has been introduced to
photoacoustic tomography in \cite{XuWan05}.
According to the notion of \cite{XuWan05}, we call
$\ubp$ the \emph{universal backprojection operator}.

\smallskip
In~\cite{XuWan05} it has been shown that the identity  $\ubp \wave f  = f$
holds for the case that $\Om$ is an open ball in three spatial dimensions and that $f$
is supported inside $\Om$.
This exact reconstruction property does not hold for general
domains. However, for arbitrarily shaped domains in arbitrary
dimensions  we have  the  following result:

\begin{theorem}\label{thm:ubp}
Let  $\Om \subset \R^n$ be a bounded convex domain with smooth boundary.
Then, for any $C^\infty$ function $f \colon \R^n \to \R$ with support in $\Om$ and any $x_0 \in \Omega$, we have
 \begin{equation} \label{eq:inv-a}
	f  \kl{x_0}
	=
	\kl{\ubp \wave f}  \kl{x_0}
	+
	\int_{\Om} f\kl{x_1} k^{(2)}_{\Om} \kl{ x_0,x_1}
	\rmd x_1 \,,
\end{equation}
with the distributional kernel
\begin{equation} \label{eq:kern-a}
k^{(2)}_{\Om} \kl{ x_0,x_1}
\coloneqq
\kl{\nabla_{x_0}+\nabla_{x_1} }^2
\int_{\Om}
\int_{\R}
\kl{\partial_t G \kl{x_1-x,t}} G\kl{x_0-x,t}
\, \rmd t \,   \rmd x \,.
\end{equation}
Here $\kl{\nabla_{x_0} + \nabla_{x_1} }^2$ is s shorthand notation for the
operator $\kl{\nabla_{x_0} + \nabla_{x_1} } \cdot \kl{\nabla_{x_0} + \nabla_{x_1} }$
and the equations \req{inv-a} and \req{kern-a} have to be read  in the
weak sense.
\end{theorem}

\begin{proof}
According to Kirchhoff's integral representation (see Lemma~\ref{lem:kirch}),
the solution $p = \wave f$  of  the initial value problem \req{wave}
satisfies
\begin{multline*}
f \kl{x_0}
=
\int_{\partial \Om}
\nu_x \cdot  \int_{\R}
 G\kl{x_0-x, t}
\nabla_{x}p \kl{x,t}
\rmd t\rmd S\kl{x}
\\
-
\int_{\partial \Om}
\nu_x \cdot
\int_{\R}
p \kl{x,t}
\nabla_{x} G\kl{x_0-x, t}
\rmd t
\rmd S\kl{x} \,.
\end{multline*}
Now inserting the identity $G \nabla_{x} p  = \nabla_{x}\kl{pG}
- p \nabla_{x} G$  in the first term  followed by an application of the
 divergence theorem, and using the relation $\nabla_{x} G\kl{x_0-x, t}  = - \nabla_{x_0} G\kl{x_0-x, t} $  yield
\begin{multline} \label{eq:aux1}
 f \kl{x_0}
=
\int_{\Om}
\int_{\R}
 \Delta_{x} \kl{p \kl{x,t}
 G\kl{x_0-x, t} }
\rmd t\rmd x
\\ +
2 \, \nabla_{x_0} \cdot
\int_{\partial \Om}\nu_x
\int_{\R}
p \kl{x,t}
G\kl{x_0-x, t}
\rmd t
\rmd S\kl{x}
\,.
\end{multline}
According to the definition of  $\ubp$,
the second term equals $\kl{\ubp \wave f} \kl{x_0}$.
After  inserting the  representation \req{sol-green} for
the solution $p = \wave f$  of  the initial value problem \req{wave} and applying the relation  $\nabla_{x} G\kl{x_i-x, t} = -
\nabla_{x_i} G\kl{x_i-x, t} $ with $i=0,1$,   the first  term in \req{aux1}
is seen to take the form
$\int_{\Om} k_\Om^{(2)} \skl{x_0,x_1} f\kl{x_1}   \rmd x_1$, with
$k_\Om^{(2)}$ defined by Equation \req{kern-a}. This concludes the proof of
Theorem~\ref{thm:ubp}.
\end{proof}

\section{Inversion in even dimension}
\label{sec:even}

Now  let $n \geq 2$ denote an even natural number.
In  this section we derive explicit formulas for the
wave inversion in even dimension and then apply these results
for  establishing Theorem \ref{thm:even}.
In even dimension, the   outgoing fundamental solution of
the wave equation  takes the following  explicit form
 \begin{equation} \label{eq:green-even}
 	G\kl{x, t} =
	\begin{cases}
 	\frac{1}{2 \pi^{n/2}}\,
 	\DD_t ^{\kl{n-2}/2}  \;
	\frac{\chi\set { t^2- \sabs{x}^2 >0 } }
	{\sqrt{t^2- \sabs{x}^2}}
	& \text{ on }  \set{t >0}   \\
	0 & \text{ on }   \set{t < 0}
 \end{cases} \,,
\end{equation}
where $\DD_t  = \kl{2t}^{-1} \partial_t$ denotes differentiation
with respect to $t^2$, and $\chi\set { t^2- \sabs{x}^2 >0 }$ is the characteristic function of the set all points $\kl{x, t} \in \R^n \times \R$
with $t^2- \sabs{x}^2 >0$.  We emphasize again that in
\req{green-even} and in similar situations all derivatives are understood as distributional derivatives.

\smallskip
We now have the  following result for recovering the initial data
of the initial value problem  \req{wave} from the restriction  of its
solution  to  $\partial \Om \times \kl{0, \infty}$.

\begin{theorem}[Wave inversion in even dimension]\label{thm:wave-even}
Let $n\geq 2$ be an  even natural number, let
$\Om \subset \R^n $ be a bounded convex domain with smooth boundary,
and let $f \colon \R^n \to \R$ be a $C^\infty$ function that is supported inside $\Om$.
Then, for every $x_0 \in \Om$,
\begin{align}\nonumber
	f\kl{x_0}
	&=
	   \kl{\K_\Om  f}\kl{x_0} + \frac{\kl{-1}^{\kl{n-2}/2}}{\pi^{n/2}} \; \times
	 \\ \label{eq:wave-even-a}
           &
            \hspace{0.17\textwidth} 	
	  \,
	\nabla_{x_0} \cdot \int_{\partial \Om}
	\nu_x
	\int_{\abs{x_0-x}}^\infty
	\frac{\mkl{t \DD_t^{\kl{n-2}/2}t^{-1}\wave f}\kl{x,t} \, \rmd t}
	{ \sqrt{t^2 - \abs{x_0-x}^2} }
	 \, \rmd S\kl{x}  \,,
\\[0.4em] \nonumber
	 f\kl{x_0}
	 &=
	 \kl{\K_\Om  f}\kl{x_0}
	 +
	 \frac{\kl{-1}^{\kl{n-2}/2}}{\pi^{n/2}}  \; \times
	 \\ \label{eq:wave-even-b}
           &
            \hspace{0.12\textwidth} 	
	\int_{\partial \Om}
	\nu_x \cdot \kl{x_0-x}
	\int_{\abs{x_0-x}}^\infty
	\frac{\mkl{\partial_{t} \DD_t^{\kl{n-2}/2}t^{-1}\wave f} \kl{x,t} \, \rmd t}
	{ \sqrt{t^2 - \abs{x_0-x}^2} }
	 \, \rmd S\kl{x}
	 \,.
\end{align}
Here $\K_\Om$, $\nu_x$, $\nabla_{x_0}$,  and $\rmd S$  are as in
Theorem~\ref{thm:even}.
\end{theorem}

We proceed this section by first deriving
Theorem~\ref{thm:wave-even} and then establishing  the corresponding
result for the inversion from spherical means
in even dimensions (namely Theorem \ref{thm:even}).

\subsection{Proof of Theorem~\ref{thm:wave-even}}

According to Theorem~\ref{thm:ubp} we have to show that
the kernel $k^{(2)}_\Om$ defined in \req{kern-a} is equal to the kernel
$k_\Om$ defined in \req{kern}, and that
$\kl{\ubp \wave f} \kl{x_0}$ can be written as  the integral term in
\req{wave-even-a} as well as the one in \req{wave-even-b}.

Let us start by  showing that $ k_\Om^{(2)} = k_\Om$,
that is,
\begin{multline} \label{eq:kern-even}
\kl{\nabla_{x_0}+\nabla_{x_1}}^2
 \int_{\Om} \int_{\R}\kl{\partial_t G\kl{x_1-x, t}} G\kl{x_0-x, t}
 \rmd t \rmd x
\\
=
\frac{\kl{-1}^{\skl{n-2}/2}}{2^{n+1}\pi^{n-1}
\abs{x_1-x_0}^{n-1}}
\kl{\partial_s^n \hilbert_s  \radon \chi_\Om}
\kl{\omega_\star , s_\star} \,,
\end{multline}
where   $\omega_\star =  \omega_\star\kl{x_0,x_1} = \frac{x_1-x_0}{\abs{x_1-x_0}}$ and
$s_\star = s_\star\kl{x_0,x_1} = \frac{\abs{x_1}^2 - \abs{x_0}^2}{2\abs{x_1-x_0}}$  are as in \req{nr}.

To show \req{kern-even},  for any two given points $x_0 \neq x_1 \in \Om$,
we write $R_0 \coloneqq \abs{x_0-x}$ and $R_1 \coloneqq \abs{x_1-x}$.
Then, using the explicit expression \req{green-even}
for the fundamental solution  of the wave equation in
even dimensions,  the inner integral on the left hand side of
 \req{kern-even} evaluates to
\begin{multline*}
\int_{\R}\kl{\partial_t G\kl{x_1-x, t}} G\kl{x_0-x, t} \rmd t
\\
\begin{aligned}
&=
\frac{1}{4 \pi^n}
\int_0^{\infty}
\kl{
\partial_t \DD_t^{\skl{n-2}/2}  \frac{\chi\set{t^2-R_1^2>0}}
{\sqrt{t^2-R_1^2}} }
\DD_t^{\skl{n-2}/2}  \frac{\chi\set{t^2-R_0^2>0}}
{\sqrt{t^2-R_0^2}} \, \rmd t
\\
&=
\frac{1}{4 \pi^n}
\int_0^{\infty}
\kl{\partial_t \DD_{R_1}^{\skl{n-2}/2}  \frac{\chi\set{t^2-R_1^2>0}}
{\sqrt{t^2-R_1^2}}
}
\DD_{R_0}^{\skl{n-2}/2}  \frac{\chi\set{t^2-R_0^2>0}}
{\sqrt{t^2-R_0^2}} \, \rmd t
\\
&=
\frac{1}{4 \pi^n}
\DD_{R_1}^{\skl{n-2}/2}
\DD_{R_0}^{\skl{n-2}/2}
\int_0^{\infty}
\kl{\partial_t   \frac{\chi\set{t^2-R_1^2>0}}
{\sqrt{t^2-R_1^2}}
}
\frac{\chi\set{t^2-R_0^2>0}}
{\sqrt{t^2-R_0^2}} \, \rmd t
\\
&=
 -
 \frac{2}{4 \pi^n}
\DD_{R_1}^{\skl{n-2}/2}
\DD_{R_0}^{\skl{n-2}/2} \;
\lim_{T\to \infty}
\DD_{R_1}
\int_{\max \set{R_0, R_1}}^{T}
\frac{t \, \rmd t}{\sqrt{t^2-R_1^2} \sqrt{t^2-R_0^2}} \,.
\end{aligned}
\end{multline*}
For $T\geq \max\set{R_0, R_1}$, the  above integral on the right hand side
computes to
\begin{multline} \label{eq:int-even}
\int_{\max\set{R_0, R_1}}^T
 \frac{t \, \rmd t}{\sqrt{t^2 - R_1^2}\sqrt{t^2 - R_0^2}}
\\
=
\ln \kl{\sqrt{T^2 - R_0^2} + \sqrt{T^2 - R_1^2}}
- \frac{1}{2}  \ln  \kl{ \abs{R_0^2 - R_1^2} }  \,.
\end{multline}
After applying the   operator  $\DD_{R_1}$ and letting
$T\to \infty$,  the first term vanishes.
Now let $\Phi \kl{s}  = 1/s$ denote the  principal value distribution
$ \varphi \mapsto \lim_{\eps \downarrow 0} \int_{\R \setminus \kl{-\eps, \eps}}
\varphi\skl{s}s^{-1}\rmd s$.
Recalling the definitions of $R_0$ and $R_1$ then implies
\begin{multline} \label{eq:aux-even}
\int_{\R}\kl{\partial_t G\kl{x_1-x, t}} G\kl{x_0-x, t} \rmd t
=
-\frac{1}{4 \pi^n} \,
\DD_{R_1}^{\skl{n-2}/2}
\DD_{R_0}^{\skl{n-2}/2}
\frac{1}{R_0^2 - R_1^2}
\\
=
\frac{\kl{-1}^{n/2}}{4 \pi^n}
\, \Phi^{\skl{n-2}} \kl{ \abs{x_0-x}^2 - \abs{x_1-x}^2}
\,.
\end{multline}
Here and in the following   $\Phi^{\skl{\nu}}$ denotes the $\nu$-th distributional
derivative of $\Phi$ for some integer number $\nu \geq 0$.

For the following recall that
$\omega_\star = \frac{x_1-x_0}{\abs{x_1-x_0}}$ and
$s_\star = \frac{\abs{x_1}^2 - \abs{x_0}^2}{2\abs{x_1-x_0}}$ and write any
point $x \in \R^n$ in the form  $x =  s \omega_\star  + y$ with $s \in \R$
and  $y \bot \omega_\star$.  We then can compute
\begin{multline} \label{eq:diffx}
\abs{x_0-x}^2 - \abs{x_1-x}^2
=
\abs{x_0}^2 -\abs{x_1}^2  + 2 x \cdot\kl{x_1-x_0}
\\
=
2 \kl{x_1-x_0} \cdot  \kl{  x - \frac{x_1 + x_0}{2} }
=
2 \abs{x_1-x_0} \kl{ s - s_\star } \,.
\end{multline}
Together with Equation~\req{aux-even} and  the definition of the Radon
transform  $\radon$, this further implies
\begin{multline*}
 \int_{\Om} \int_{\R}\kl{\partial_t G\kl{x_1-x, t}} G\kl{x_0-x, t}
 \rmd t \rmd x
\\
\begin{aligned}
& =
\frac{\kl{-1}^{n/2}}{4 \pi^n}
\int_{\R}
\int_{\omega_\star^\bot}
\chi_\Om\kl{s \omega_\star + y }
\Phi^{(n-2)} \kl{ 2 \abs{x_1-x_0} \kl{ s - s_\star }}
\, \rmd y \, \rmd s
\\
&=
\frac{\kl{-1}^{n/2}}{4 \pi^n}
\int_{\R}
\Phi^{(n-2)} \kl{ 2 \abs{x_1-x_0} \kl{ s - s_\star }}
\kl{\int_{\omega_\star^\bot}
\chi_\Om\kl{s \omega_\star + y }
\, \rmd{y} }  \rmd s
\\
&=
\frac{\kl{-1}^{n/2}}{4 \pi^n}
\int_{\R}
\Phi^{(n-2)} \kl{ 2 \abs{x_1-x_0} \kl{ s - s_\star }}
\kl{\radon \chi_\Om}\kl{\omega_\star , s} \rmd s
\,.
\end{aligned}
\end{multline*}

Now using the chain rule,
integrating by parts $n-2$ times,
recalling the definition of the principal value distribution $\Phi \kl{s} = 1/s$, and
noting that the Hilbert transform $\hilbert_s$ is defined as the convolution with $\pi^{-1} \Phi$
imply
\begin{multline*}
 \int_{\Om} \int_{\R}\kl{\partial_t G\kl{x_1-x, t}} G\kl{x_0-x, t}
 \rmd t \rmd x
\\
\begin{aligned}
&=
\frac{\kl{-1}^{n/2}}{4 \pi^n 2^{n-2} \abs{x_1-x_0}^{n-2}}
\int_{\R}
\kl{\partial_s^{n-2}
\Phi \kl{ 2 \abs{x_1-x_0} \kl{ s - s_\star }} }
\kl{\radon \chi_\Om}\kl{\omega_\star , s} \rmd s
\\
&=
\frac{\kl{-1}^{n/2}}{\pi^n 2^{n} \abs{x_1-x_0}^{n-2}}
\int_{\R}
\frac{1}{2 \abs{x_1-x_0} \kl{ s - s_\star }}
\kl{\partial_s^{n-2}  \radon \chi_\Om}\kl{\omega_\star , s} \rmd s
\\
&=
\frac{\kl{-1}^{\kl{n-2}/2}}{2^{n+1} \pi^n  \abs{x_1-x_0}^{n-1}}
\int_{\R}
\frac{1}{s_\star - s }
\kl{\partial_s^{n-2}  \radon \chi_\Om}\kl{\omega_\star , s} \rmd s
\\
&=
\frac{\kl{-1}^{\skl{n-2}/2}}{2^{n+1}\pi^{n-1}
\abs{x_1-x_0}^{n-1}}
\kl{\partial_s^{n-2} \hilbert_s \radon \chi_\Om}
\kl{\omega_\star , s_\star} \,.
\end{aligned}
\end{multline*}

It remains to apply the operator $\kl{\nabla_{x_0}+\nabla_{x_1}}^2$
to the last expression. To that end, notice that due to symmetry
$\nabla_{x_0}+\nabla_{x_1}$ applied to any distribution  only depending on  $x_0-x_1$ vanishes,
and that $\kl{\nabla_{x_0}+\nabla_{x_1}} s_\star =  \kl{x_1-x_0}/\abs{x_1-x_0}$.
This implies
\begin{multline*}
\kl{\nabla_{x_0}+\nabla_{x_1}}^2
 \int_{\Om} \int_{\R}\kl{\partial_t G\kl{x_1-x, t}} G\kl{x_0-x, t}
 \rmd t \rmd x
\\
= \frac{\kl{-1}^{\skl{n-2}/2}}{2^{n+1}\pi^{n-1}
\abs{x_1-x_0}^{n-1}}
\kl{\partial_s^n \hilbert_s  \radon \chi_\Om}
\kl{\omega_\star , s_\star} \,,
\end{multline*}
which is the equality claimed in \req{kern-even}.

Now recall the definition of $\ubp$  (see Equation~\req{ubp})
as well as the explicit representation \req{green-even} for the fundamental solution
of the wave equation  in even dimension. Further notice that
for any  integer $\nu$, the formal $L^2$ adjoint of  $\DD_t^{\nu}$
is given by $\kl{\DD_t^{\nu}}^* = \kl{-1}^{\nu}  t \DD_t^{\nu} t^{-1}$.
We therefore  can compute
\begin{multline*}
\kl{\ubp\wave f} \kl{x_0}
\\
\begin{aligned}
& =
\frac{1}{\pi^{n/2}}  \;
	\nabla_{x_0} \cdot \int_{\partial \Om}
	\nu_x
	\int_{0}^\infty
	\DD_t^{\kl{n-2}/2} \kl{\frac{\chi\set{t^2 >  \sabs{x_0-x}^2}}{\sqrt{t^2 - \abs{x_0-x}^2}}}
	\wave f \kl{x, t}
	 \, \rmd t \, \rmd S\kl{x}
\\
& =
\frac{\kl{-1}^{\kl{n-2}/2}}{\pi^{n/2}}  \;
	\nabla_{x_0} \cdot \int_{\partial \Om}
	\nu_x
	\int_{0}^\infty
	\frac{\chi\set{t^2 >  \sabs{x_0-x}^2}}{\sqrt{t^2 - \abs{x_0-x}^2}}
	\kl{t \DD_t^{\kl{n-2}/2}  t^{-1}
	\wave f }\kl{x, t}
	 \, \rmd t \, \rmd S\kl{x} \\
& =
\frac{\kl{-1}^{\kl{n-2}/2}}{\pi^{n/2}} \;
	\nabla_{x_0} \cdot \int_{\partial \Om}
	\nu_x
	\int_{\abs{x_0-x}}^\infty
	\frac{\skl{t \DD_t^{\kl{n-2}/2}
	t^{-1}
	\wave f } \kl{x, t}}
	{\sqrt{t^2 - \abs{x_0-x}^2}}
	 \, \rmd t \, \rmd S\kl{x}  \,.
\end{aligned}
\end{multline*}
In fact, the second equality follows from repeated integration by parts.
The boundary terms at $\infty$ vanish since, due to the compact support of
$f$,  all derivatives of  $\wave f \kl{x,t}$ tend to zero as $t \to \infty$
(uniformly with respect to  $x$).
In view of Theorem~\ref{thm:ubp} and
Equation \req{kern-even}, the above expression for
$\kl{\ubp\wave f} \kl{x_0}$  yields the first identity in Theorem
\ref{thm:wave-even}, formula~\req{wave-even-a}.

Finally, we verify the second identity in Theorem
\ref{thm:wave-even}.   Interchanging  the order of differentiation in the
last displayed expression for $\kl{\ubp\wave f} \kl{x_0}$ followed by one
integration by parts yields
\begin{multline*}
\kl{-1}^{\kl{n-2}/2}\pi^{-n/2}
\kl{\ubp\wave f} \kl{x_0}
\\
\begin{aligned}
& =
-	 \int_{\partial \Om}
	\nu_x \cdot \kl{x_0-x}
	\int_{0}^\infty
	\partial_t\kl{\frac{ \chi \set{t^2  > \sabs{x_0-x}^2}}
	{ \sqrt{t^2 - \abs{x_0-x}^2} }}
	 \DD_t^{\kl{n-2}/2}  t^{-1}
	\wave f \kl{x, t}
	 \rmd t \rmd S\kl{x}
\\ &=
\int_{\partial \Om}
	\nu_x \cdot \kl{x_0-x}
	\int_{\sabs{x_0-x}}^\infty
	\frac{  \skl{\partial_t\DD_t^{\kl{n-2}/2}  t^{-1}
	\wave f} \kl{x, t} }
	{ \sqrt{t^2 - \abs{x_0-x}^2} }
	 \, \rmd t \, \rmd S\kl{x}
	 \,.
\end{aligned}
\end{multline*}
This shows equation \req{wave-even-b} and
concludes the proof of Theorem~\ref{thm:wave-even}.

\subsection{Proof of  Theorem \ref{thm:even}}

Recall the formula \req{sol-green} for the solution of
the wave equation \req{wave} as well as  the explicit
expression \req{green-even} for the fundamental  solution of the wave equation in
even dimensions. After introducing polar coordinates around the center
$x \in \partial \Om$ we can write
\begin{align*}
\kl{\wave f} \kl{x, t }
&=
\int_{\Om} \kl{\partial_t G\kl{x_1-x, t} } f\kl{x_1} \rmd x_1
\\
&=
\frac{1}{2\pi^{n/2}}
\int_{\Om} \kl{\partial_t \DD_t^{\skl{n-2}/2}
\frac{\chi\set{t^2 -\sabs{x_1-x}^2 > 0}}
{\sqrt{t^2 -\abs{x_1-x}^2}}
}   f\kl{x_1} \rmd x_1
\\
&=
\frac{\om_{n-1}}{2\pi^{n/2}}
\int_{0}^\infty
r^{n-1}  \M f \kl{x, r}
\kl{\partial_t \DD_t^{\skl{n-2}/2}
\frac{\chi\set{t^2 -r^2 > 0}}
{\sqrt{t^2 -r^2}}
}
\rmd r
\,.
\end{align*}

After multiplying  the last displayed equation with $G\kl{x_0-x, t}$,
integrating over  the time variable  and
using the shorthand notation  $R_0  =  \sabs{x_0-x}$ we obtain
\begin{multline*}
 \int_{\R}
 	G\kl{x_0-x, t}
 	\wave f \kl{x, t} \rmd t
	\\
	\begin{aligned}
	& =
\frac{\om_{n-1}}{4\pi^{n}}
\int_0^\infty
\kl{
\DD_t^{\skl{n-2}/2}
\frac{\chi\set{t^2 > R_0^2 }}
{\sqrt{t^2 -R_0^2}}} \, \times
\\& \hspace{0.18\textwidth}
\int_{0}^\infty
r^{n-1}  \M f \kl{x, r}
\kl{ \partial_t \DD_t^{\skl{n-2}/2}
\frac{\chi\set{t^2>r^2 }}
{\sqrt{t^2 -r^2}}
} \rmd r \rmd t
\\[0.2em]
& =
\frac{\om_{n-1}}{4\pi^{n}}
\int_{0}^\infty
r^{n-1}  \M f \kl{x, r}\, \times
\\& \hspace{0.06\textwidth}
\int_0^\infty
\kl{ \DD_t^{\skl{n-2}/2}
\frac{\chi\set{t^2> R_0^2}}
{\sqrt{t^2 -R_0^2}}}
\kl{
\partial_t \DD_t^{\skl{n-2}/2}
\frac{\chi\set{t^2  >r^2 }}
{\sqrt{t^2 -r^2}}
} \rmd t\rmd r
\\[0.2em]
& =
-\frac{\om_{n-1}}{4\pi^{n}}
\int_{0}^\infty
r^{n-1}  \M f \kl{x, r}
\DD_{R_0}^{\skl{n-2}/2}
\DD_{r}^{\skl{n-2}/2}
\, \times
\\
& \hspace{0.26\textwidth}
\kl{
\DD_{r}
\lim_{T \to \infty}\int_{\max\set{R_0,r}}^T
\frac{ 2 t \rmd t}
{\sqrt{t^2 -R_0^2} \sqrt{t^2 -r^2}}
} \rmd r \,.
\end{aligned}
\end{multline*}
The inner integral has already been computed
(see Equation \req{int-even}) and shows
\begin{equation*}
\DD_{r}
\lim_{T \to \infty}\int_{\max\set{R_0,r}}^T
\frac{ 2 t \rmd t}
{\sqrt{t^2 -R_0^2} \sqrt{t^2 -r^2}}
=- \frac{1}{r^2 - R_0^2}
= -\Phi \kl{r^2 - R_0^2}  \,,
\end{equation*}
with $\Phi$ denoting the principal value distribution $\mathrm{P.V.}~1/s$.
After recalling that $\DD_r$ denotes differentiation with  respect to $r^2$ and that
the formal $L^2$ adjoint of  $\DD_r^{\nu}$ is given by $\kl{\DD_r^{\nu}}^* =
\kl{-1}^{\nu}  r \DD_r^{\nu} r^{-1}$ we obtain
\begin{multline*}
 \int_{\R}
 	G\kl{x_0-x, t}
 	\wave f \kl{x, t} \rmd t
	\\
	\begin{aligned}
	&=
	\frac{\kl{-1}^{\skl{n-2}/2} \om_{n-1}}
	{4\pi^{n}}
	\int_{0}^\infty
	r^{n-1}
	\M f \kl{x, r}
	\DD_{r}^{n-2} \Phi\kl{ r^2 - \abs{x_0-x}^2 } \rmd r
	\\
	&=
	\frac{\kl{-1}^{\skl{n-2}/2} \om_{n-1}}
	{4\pi^{n}}
	\int_{0}^\infty
	\frac{\skl{r \DD_{r}^{n-2} r^{n-2}
	\M f }\kl{x, r} }{r^2 - \abs{x_0-x}^2}
	\; \rmd r \,.
\end{aligned}
\end{multline*}
Finally, by  using the definition of  $\ubp$ (see Equation~\req{ubp})
and inserting the identity just established we obtain
\begin{multline} \label{eq:wave-even-aux}
\kl{\ubp \wave f} \kl{x_0}
=
2  \, \nabla_{x_0} \cdot \int_{\partial \Om} \nu_x
 \int_{\R}
 G\kl{x_0-x, t}
 \wave f \kl{x, t} \rmd t \rmd S \kl{x}
\\
=
\frac{\kl{-1}^{\skl{n-2}/2} \om_{n-1}}
	{2\pi^{n}}
	\;
	\nabla_{x_0} \cdot \int_{\partial \Om} \nu_x
\int_{0}^\infty
	\frac{\skl{r \DD_{r}^{n-2} r^{n-2}
	\M f }\kl{x, r} }{r^2 - \abs{x_0-x}^2}  \, \rmd r \rmd S \kl{x}\,.
\end{multline}
According to Theorem~\ref{thm:ubp} and
Equation \req{kern-even} this shows the first
inversion formula~\req{even-a} in Theorem~\ref{thm:even}.

\smallskip
It remains to  verify  Equation \req{even-b}.
This equation, however, is an easy consequence  of the identity \req{even-a}
just established. In fact, interchanging the order of integration and
differentiation in \req{wave-even-aux}  and integrating by parts
yields
\begin{multline*}
\kl{ \ubp \wave f} \kl{x_0}
=
\frac{\kl{-1}^{\skl{n-2}/2} \om_{n-1}}
	{2\pi^{n}} \,
	\times
	 \\
	\int_{\partial \Om} \nu_x  \cdot \kl{x_0-x}
\int_{0}^\infty
	\frac{\skl{\partial_{r} \DD_{r}^{n-2} r^{n-2}
	\M f }\kl{x, r} }{r^2 - \abs{x_0-x}^2}   \, \rmd r \rmd S \kl{x} \,.
\end{multline*}
Again, according to Theorem~\ref{thm:ubp} and
Equation~\req{kern-even} the last displayed equation
implies~\req{even-b} and concludes the proof of Theorem~\ref{thm:even}.

\section{Inversion in odd dimension}
\label{sec:odd}

Now let $n \geq 3$ be an odd natural number.
In this case, the outgoing fundamental solution
of the wave equation is given
\begin{equation} \label{eq:green-odd}
G\kl{x, t}
=
\begin{cases}
\frac{1}{2 \pi^{\kl{n-1}/2}} \;  \DD_t^{\kl{n-3}/2}
\delta \kl{t^2 - \abs{x}^2 }
& \text{ on } \set{t > 0 }
\\
0 & \text{ on } \set{t < 0 }
\end{cases} \,.
\end{equation}
Here, as usual,  the operators
$\DD_t  = \kl{2t}^{-1} \partial_t$ denotes
the distributional derivative  with respect to the
variable $t^2$.

\smallskip
We have the following counterpart of
Theorem~\ref{thm:wave-even} for odd
dimensions.

\begin{theorem}[Wave inversion in odd dimension]\label{thm:wave-odd}
Let $n \geq 3$ be an odd natural number,
let $\Om \subset \R^n $ be a  bounded convex domain with smooth boundary,
and let $f \colon \R^n \to \R$ be a $C^\infty$ function that is supported inside $\Om$.

Then, for every $x_0 \in \Om$, we have
\begin{align}\nonumber
	f\kl{x_0}
	&=
	   \kl{\K_\Om  f}\kl{x_0}\;
	 \\
	 \label{eq:wave-odd-a}
           &
            +
	\frac{\kl{-1}^{\kl{n-3}/2}}{2\pi^{\kl{n-1}/2}} \,
	\nabla_{x_0} \cdot \int_{\partial \Om}
	\nu_x
	\, \kl{\DD_t^{\kl{n-3}/2}t^{-1}\wave f }
	\kl{x,\sabs{x_0-x}}
	 \, \rmd S\kl{x}  \,,
\\[0.4em] \nonumber
		f\kl{x_0}
	&=
	   \kl{\K_\Om  f}\kl{x_0}\;
	 \\ \label{eq:wave-odd-b}
           &
            +
	\frac{\kl{-1}^{\kl{n-3}/2}}{2\pi^{\kl{n-1}/2}}  \,
	\int_{\partial \Om}
	\nu_x \cdot \frac{x_0-x}{\sabs{x_0-x}}
	\, \kl{ \partial_t \DD_t^{\kl{n-3}/2}t^{-1}\wave f }
	\kl{x,\sabs{x_0-x}}
	 \, \rmd S\kl{x}
	 \,.
\end{align}
Here, again, $\K_\Om$, $\nu_x$, $\nabla_{x_0}$,  and $\rmd S$  are as in
Theorem~\ref{thm:even}.
\end{theorem}

We proceed with this section by first  establishing
Theorem~\ref{thm:wave-odd} and then deriving
the formulas in  Theorem~\ref{thm:odd} as corollaries
of it.

\subsection{Proof of Theorem~\ref{thm:wave-odd}}

Similar  to  the even dimensional case we apply Theorem~\ref{thm:ubp}
and verify  that the kernel $k^{(2)}_\Om$ defined in \req{kern-a} is  equal to the
kernel $k_\Om$ defined in   \req{kern},
and that  $\kl{\ubp \wave f} \kl{x_0}$ can be written as any of the integral terms in
Equations~\req{wave-odd-a} and \req{wave-odd-b}.

We first show that $k_\Om^{\kl{2}} = k_\Om$, that is,
\begin{multline} \label{eq:kern-odd}
\kl{\nabla_{x_0} + \nabla_{x_1}}^2
 \int_{\Om} \int_{\R}\kl{\partial_t G\kl{x_1-x, t}} G\kl{x_0-x, t}
 \rmd t \rmd x
\\
=
\frac{\kl{-1}^{\kl{n-1}/2}}{2^{n+1}\pi^{n-1}
\abs{x_1-x_0}^{n-1}}
\kl{\partial_s^{n}  \radon \chi_\Om}
\kl{\omega_\star , s_\star} \,,
\end{multline}
where   $\omega_\star =  \omega_\star\kl{x_0,x_1} = \frac{x_1-x_0}{\abs{x_1-x_0}}$ and
$s_\star = s_\star\kl{x_0,x_1} = \frac{\abs{x_1}^2 - \abs{x_0}^2}{2\abs{x_1-x_0}}$  are as in \req{nr}.

With the notation $R_0 \coloneqq \abs{x_0-x}$ and
$R_1 \coloneqq \abs{x_1-x}$,  the
representation~\req{green-odd} for the outgoing fundamental
solution of the wave equation in odd dimensions  yields
\begin{multline*}
\int_{\R} \kl{\partial_t G\kl{x_1-x, t}}  G\kl{x_0-x, t} \rmd t
\\
\begin{aligned}
&=
\frac{1}{4 \pi^{n-1}}
\int_{0}^\infty
\kl{ \partial_t \DD_t^{\skl{n-3}/2}  \delta\kl{t^2 - R_1^2} }
\DD_t^{\skl{n-3}/2}  \delta\kl{t^2 - R_0^2}
 \rmd t
\\
&=
\frac{2}{4 \pi^{n-1}}
\int_{0}^\infty
\kl{t \DD_t^{\skl{n-1}/2}  \delta\kl{t^2 - R_1^2}}
\DD_t^{\skl{n-3}/2}  \delta\kl{t^2 - R_0^2}
 \rmd t
\\
&=
 -
 \frac{2}{4 \pi^{n-1}}
\int_{0}^\infty
\kl{ t \DD_{R_1}^{\skl{n-1}/2}  \delta\kl{t^2 - R_1^2} }
\DD_{R_0}^{\skl{n-3}/2}  \delta\kl{t^2 - R_0^2}
 \rmd t
\\
&=
-
\frac{1}{4 \pi^{n-1}} \;
\DD_{R_1}^{\skl{n-1}/2}
\DD_{R_0}^{\skl{n-3}/2} \int_{0}^\infty
\delta\kl{t^2 - R_1^2}
\delta\kl{t^2 - R_0^2}
 \, 2 t \, \rmd t
\\
&=
-
 \frac{1}{4 \pi^{n-1}} \;
\DD_{R_1}^{\skl{n-1}/2}
\DD_{R_0}^{\skl{n-3}/2}
 \delta\kl{R_0^2 - R_1^2}
 \\
&=
\frac{\kl{-1}^{\skl{n-3}/2}}{4 \pi^{n-1}} \;
 \delta^{(n-2)}\kl{\abs{x_0-x}^2 -\abs{x_1-x}^2}
 \,.
\end{aligned}
\end{multline*}
Here and in the following   $\delta^{(\nu)}$ denotes the $\nu$-th derivative of the
one-dimensional  delta distribution for some integer number $\nu \geq 0$.

Now recall the definitions
$\omega_\star =   \frac{x_1-x_0}{\abs{x_1-x_0}}$ and
$s_\star =  \frac{\abs{x_1}^2 - \abs{x_0}^2}{2\abs{x_1-x_0}}$ and  write $x =  s \omega_\star  + y$ with $s \in \R$
and  $y \bot \omega$. We then have (see Equation~\req{diffx})
\begin{equation*}
\abs{x_0-x}^2 - \abs{x_1-x}^2
=
2 \abs{x_1-x_0} \kl{s-s_\star} \,.
\end{equation*}
This   implies
 \begin{multline*}
\int_{\Om}  \int_{\R}\kl{\partial_t G\kl{x_1-x, t}} G\kl{x_0-x, t}
\rmd t \rmd x
\\
\begin{aligned}
&=
 \frac{\kl{-1}^{\skl{n-3}/2}}{4 \pi^{n-1}} \;
\int_{\R}
\int_{\omega_\star^\bot}
\chi_\Om\kl{s \omega_\star + y }
\delta^{(n-2)} \kl{ 2 \abs{x_1-x_0} \kl{ s - s_\star }}
\rmd y \rmd s
\\
&=
\frac{\kl{-1}^{\kl{n-3}/2}}{4 \pi^{n-1}}
\int_{\R}
\delta^{(n-2)} \kl{ 2 \abs{x_1-x_0} \kl{ s - s_\star }}
\kl{\int_{\omega_\star^\bot}
\chi_\Om\kl{s \omega_\star + y }
\rmd y }\rmd s
\\
&=
\frac{\kl{-1}^{\kl{n-3}/2}}{4\pi^{n-1}}
\int_{\R}
\delta^{(n-2)} \kl{ 2 \abs{x_1-x_0} \kl{ s - s_\star }}
\kl{\radon \chi_\Om}\kl{\omega_\star , s} \rmd s
\,.
\end{aligned}
\end{multline*}
Integrating   $n-2$ times by parts yields
\begin{multline*}
\int_{\Om}  \int_{\R}\kl{\partial_t G\kl{x_1-x, t}} G\kl{x_0-x, t}
\rmd t \rmd x
\\
\begin{aligned}
&=
\frac{\kl{-1}^{\kl{n-1}/2}}{4 \pi^{n-1}}
\frac{1}{2^{n-2} \abs{x_1-x_0}^{n-2} }
\int_{\R}
\delta \kl{ 2 \abs{x_1-x_0} \kl{ s - s_\star }}
\kl{\partial_s^{n-2} \radon \chi_\Om}\kl{\omega_\star , s} \rmd s
\\
&=
\frac{\kl{-1}^{\kl{n-1}/2}}{2^{n+1}
\pi^{n-1}  \abs{x_1-x_0}^{n-1}}
\, \kl{\partial_s^{n-2} \radon \chi_\Om}\kl{\omega_\star ,  s_\star}
\,.
\end{aligned}
\end{multline*}
As in the even dimension case, after application
of $\kl{\nabla_{x_0}+\nabla_{x_1}}^2$  this yields \req{kern-odd}.

Next, note that the fundamental solution \req{green-odd}  in  odd dimensions may be rewritten in the form
$1/\skl{4\pi^{\skl{n-1}/2}} \DD_t^{\kl{n-3}/2}
t^{-1} \delta \kl{t - \abs{x}}$.  Consequently,  by  the definition of  $\ubp$ (see Equation~\req{ubp}),
we have
\begin{align*}
\kl{\ubp\wave f}\kl{x_0}
&=
2\, \nabla_{x_0} \cdot \int_{\partial \Om} \nu_x \int_{\R}
G \kl{x- x_0, t} \wave f \kl{x, t } \rmd t \rmd S \kl{x}
\\
&=
\frac{\kl{-1}^{\kl{n-3}/2}}{2\pi^{\kl{n-1}/2}} \;  \nabla_{x_0} \cdot
\int_{\partial \Om} \nu_x \,
\skl{\DD_t^{\kl{n-3}/2}
t^{-1}
\wave f} \kl{x, \abs{x_0-x} } \rmd S \kl{x} \,.
\end{align*}
In view of Theorem~\ref{thm:ubp}  and due to
Equation~\req{kern-odd} this implies  formula~\req{wave-odd-a}
claimed in Theorem \ref{thm:wave-odd}.
Finally, carrying out  the  differentiation under the integral
yields
\begin{multline*}
\kl{\ubp\wave f}\kl{x_0}
\\=
\frac{\kl{-1}^{\kl{n-3}/2}}{2\pi^{\kl{n-1}/2}} \;
\int_{\partial \Om} \nu_x  \cdot \frac{x_0-x}{\sabs{x_0-x}}
\kl{\partial_t \DD_t^{\kl{n-3}/2}
t^{-1} \wave f } \kl{x, \abs{x_0-x} }
\rmd S \kl{x} \,.
\end{multline*}
This shows that also identity \req{wave-odd-b} holds and
concludes the proof of Theorem \ref{thm:wave-odd}.

\subsection{Proof of Theorem~\ref{thm:odd}}

Inserting the expression \req{green-odd} for the fundamental solution of the wave equation in odd
dimensions (see Equation \req{sol-green}) shows that the solution of  the wave equation~\req{wave}
can be written as
\begin{equation*}
\wave f \kl{x, r}
=
\frac{\om_{n-1}}{4 \pi^{\kl{n-1}/2}} \;
\partial_{r} \DD_r^{\kl{n-3}/2}
r^{n-2} \M f \kl{x, r} \,.
\end{equation*}
Together with the definition of  $\ubp$ (see Equation \req{ubp}) this yields
\begin{multline*}
\kl{\ubp \wave f}  \kl{x_0}
\\
\begin{aligned}
&=
\frac{1}{2 \pi^{\kl{n-1}/2}} \;
\nabla_{x_0} \cdot \int_{\partial \Om}
\nu_x
\kl{ \DD_r^{\kl{n-3}/2} r^{-1} \wave f} \kl{x, \abs{x_0-x}} \rmd S\kl{x}
\\
&=
\frac{\om_{n-1}}{8 \pi^{n-1}} \;
\nabla_{x_0} \cdot \int_{\partial \Om}
\nu_x
\kl{ \DD_r^{\kl{n-3}/2} r^{-1}\partial_{r} \DD_r^{\kl{n-3}/2}
r^{n-2} \M f } \kl{x, \abs{x_0-x}} \rmd S\kl{x}
\\
&=
\frac{\om_{n-1}}{4 \pi^{n-1}} \;
\nabla_{x_0} \cdot \int_{\partial \Om}
\nu_x
\kl{ \DD_r^{n-2}  r^{n-2} \M f }
\kl{x, \abs{x_0-x}} \rmd S\kl{x} \,.
\end{aligned}
\end{multline*}
According to  Theorem~\ref{thm:ubp}  and Equation~\req{kern-odd}
this yields~\req{odd-a}.

\smallskip
It remans to establish the second formula in Theorem \ref{thm:odd},
namely Equation ~\req{odd-b}.
To that end, one simply carries out the differentiation in the last displayed formula for
$\kl{\ubp \wave f } \kl{x_0}$, which yields
\begin{align*}
\kl{\ubp \wave f  } \kl{x_0}
&=
\frac{\om_{n-1}}{4 \pi^{n-1}} \;
\nabla_{x_0} \cdot \int_{\partial \Om}
\nu_x
\kl{ \DD_r^{n-2}  r^{n-2} \M f }
\kl{x, \abs{x_0-x}} \rmd S\kl{x}
\\
&=
\frac{\om_{n-1}}{4 \pi^{n-1}} \;
\int_{\partial \Om}
\nu_x \cdot \frac{x_0-x}{\sabs{x_0-x}}
\kl{ \partial_{r} \DD_r^{n-2}  r^{n-2} \M f }
\kl{x, \abs{x_0-x}} \rmd S\kl{x} \,.
\end{align*}
This however yields formula~\req{odd-b}.

\section{Exact inversion for elliptical domains}
\label{sec:ell}

Let $A = \diag \kl{a_1, \dots, a_n}$ be a diagonal
matrix in   $\R^{n \times n}$ with entries $a_j >0$
and consider the elliptical domain
\begin{equation*}
	\Om
	\coloneqq  \set{ x \in \R^n :
	\abs{ A^{-1} x }^2  < 1 } \,.
\end{equation*}
In order to establish the exact inversion formulas of Theorem~\ref{thm:ell}
it is sufficient  to show that $\K_\Om f = 0$.  This will be done by
first verifying that the kernel vanishes  for the special case that  the domain is a ball
and then applying  a linear transformation to the ball to establish the result for general case.

\subsection{Special case: spherical domains}
 \label{sec:ell-ball}

Let $B \coloneqq   \set{ x \in \R^n  :  \sabs{x} < 1 }$
denote the unit ball in $\R^n$ centered at the origin.
Then,  elementary geometry shows that the Radon transform
of $\chi_{B}$ is given by
\begin{equation} \label{eq:radon-ball}
\radon \chi_B \kl{\omega, s }
=
\begin{cases}
V_{n-1}  \kl{1 - s^2}^{\kl{n-1}/2} & \text{ if } \abs{s} < 1 \\
0 & \text{otherwise}
\end{cases}\,,
\end{equation}
where $V_{n-1} = \frac{\om_{n-2}}{n-1}$
denotes the volume of the unit ball in $\R^{n-1}$.

\subsubsection*{Odd  dimension}

If  $n$ is odd, then \req{radon-ball} shows
that $\radon \chi_B \kl{\omega, s }$ is a polynomial of degree $n-1$ on
$\set{\abs{s} < 1}$.
Since  $ \abs{s_\star\kl{x_0,x_1}} < 1$ for
any two distinct points  $x_0, x_1 \in B$,
this yields
\begin{equation*}
k_B \kl{x_0,x_1} =
\frac{\kl{-1}^{\kl{n-1}/2}}{2^{n+1}\pi^{n-1}}
	 \,
	 \frac{\kl{ \partial_s^n  \radon \chi_B}
	\kl{\omega_\star\kl{x_0,x_1},  s_\star\kl{x_0,x_1}}}
	{\abs{x_1-x_0}^{n-1}}
	= 0 \; \text{ for } x_1 \neq x_0 \in B \,.
\end{equation*}
This implies  that we also have $\K_B f = 0$ and, according to
Theorem~\ref{thm:odd}, this shows the inversion formulas~\req{inv-ell-odd-a}, \req{inv-ell-odd-b}
stated in Theorem~\ref{thm:ell} for the special case that  the considered
domain is a ball in odd dimension.

\subsubsection*{Even dimension}

If $n \geq 2$ is an even natural number, then
the identity $\K_B f = 0$   is slightly less obvious.
In this case, we first note the following identity (see, for example,  \cite[Table~7.3, Number~13]{Pou10})
\begin{equation}\label{eq:hp}
\kl{\hilbert_s \kl{ s  \phi }}\kl{\hat s} = \hat s  \kl{\hilbert_s \phi} \kl{\hat s} \ -
\frac{1}{\pi} \int_{\R} \phi \kl{s} \rmd s
\end{equation}
satisfied by the Hilbert transform and some function $\phi \colon \R \to \R$.
Further,   Equation~\req{radon-ball} shows that
we have the relation $\kl{\radon \chi_B} \kl{\omega, s }  =  P_{n-2} \kl{s}
\phi_{1/2} \kl{s} $, where $P_{n-2} \kl{s}$ is a polynomial of degree
$n-2$ and $\phi_{1/2}\kl{s} \coloneqq \sqrt{\max \set{0,  1 - s^2}}$.
Applying the identity   \req{hp} repeatedly, thus yields
\begin{equation*}
\kl{\hilbert_s \radon \chi_B } \kl{\omega, s }
=
Q_{n-2} \kl{s}  \kl{\hilbert_s \phi_{1/2}} \kl{s}
+
Q_{n-3} \kl{s} \,,
\end{equation*}
for certain polynomials  $Q_{n-2}$ and $Q_{n-3}$ of degree $n-2$ and $n-3$, respectively.  The Hilbert transform of
$\phi_{1/2}\kl{s} =  \sqrt{\max \set{0,  1 - s^2}}$ is  known and
given by (see, for example,  \cite[Table 13.11]{Bra00b})
\begin{equation*}
\kl{\hilbert_s  \phi_{1/2} } \kl{s}
=
s -
\operatorname{sign} \kl{s}
\chi \set{\abs{s} > 1} \sqrt{s^2-1}
\quad \text{ for all } s \in \R
\,.
\end{equation*}
In particular, $\kl{\hilbert_s  \phi_{1/2}}\kl{s}$  is a   linear function  on
$ \set{\abs{s} < 1}$ and therefore the product $Q_{n-2} \kl{s}  \skl{\hilbert_s \phi_{1/2}} \kl{s}$ is a polynomial of degree $n-1$
on $\set{\abs{s} < 1}$. Noting again that
$ \abs{s_\star\kl{x_0,x_1}} < 1$ for any two distinct points
$x_0, x_1 \in B$, we therefore conclude
\begin{equation*}
k_B \kl{x_0,x_1} =
	  \frac{\kl{-1}^{\kl{n-2}/2}}{2^{n+1} \pi^{n-1}}
	  \,
	  \frac{\kl{ \partial_s^n \hilbert_s  \radon \chi_B}
	  \kl{\omega_\star\kl{x_0,x_1},  s_\star\kl{x_0,x_1}}}
	  {  \abs{x_1-x_0}^{n-1}}
	  = 0 \; \text{ for } x_1 \neq x_0 \in B \,.
\end{equation*}
This implies   that the identity  $\K_B f = 0$ also holds in  even dimension.
In view of Theorem~\ref{thm:even}, this establishes
inversion formulas \req{inv-ell-even-a}, \req{inv-ell-even-b} in Theorem \ref{thm:ell}
for the special case that  the considered domain is a ball  in even dimension.

\subsection{General case: elliptical domains}

Now let $ \Om = \set{ x \in \R^n : \sabs{A^{-1}x} < 1}$ be an
elliptical domain where $A = \diag \kl{a_1, \dots, a_{n}}$ is a
diagonal matrix with positive entries that  possibly differ from
each other. We then obviously have the identity
$\chi_\Om  \kl{x} = \chi_B \kl{ A^{-1}x} $,
where  $B \subset \R^n$ is the unit ball considered in the previous subsection.
Therefore,  the known relation between the Radon transform of  a function $\varphi$ and the Radon transform of
the function $ x \mapsto \varphi \kl{ A^{-1}x }$ implies that
\begin{equation} \label{eq:rbo}
	\radon \chi_\Om \kl{\omega, s}
	=
	\frac{\det \kl{A}}{\abs{A \omega}} \;
	\radon \chi_B \kl{\frac{A \omega}{\abs{A \omega}}, \frac{s} {\abs{A \omega}}}
	\qquad  \text{ for all }
	\kl{\omega, s} \in S^{n-1} \times \R  \,,
 \end{equation}
From \req{rbo} we  conclude that
 \begin{equation*}
\begin{aligned}
 \kl{\partial_s^{n}
 \radon \chi_\Om} \kl{\omega, s}
 &=  \frac{\det \kl{A}}{\abs{A \omega}^{n+1}} \;
 \kl{\partial_s^{n}
 \radon \chi_B } \kl{\frac{A\omega}{\abs{A \omega}}, \frac{s} {\abs{A \omega}}}
&& \text{if  $n$ is odd} \,,
\\
\kl{ \partial_s^n \hilbert_s
 \radon \chi_\Om} \kl{\omega, s}
 &=  \frac{\det \kl{A}}{\abs{A \omega}^{n+1}} \;
 \kl{\partial_s^n \hilbert_s
 \radon \chi_B} \kl{\frac{A\omega}{\abs{A \omega}}, \frac{s} {\abs{A \omega}}}
&& \text{if  $n$ is even} \,.
 \end{aligned}
 \end{equation*}
According to the special case considered in Subsection \ref{sec:ell-ball}
this shows that  $k_\Om \kl{x_0,x_1} = 0$ for all $x_0 \neq x_1 \in \Om$ and hence
that $\K_\Om f = 0$.
In view of Theorems~\ref{thm:even} and \ref{thm:odd}, this   establishes
the exact  reconstruction formulas \req{inv-ell-even-a}, \req{inv-ell-even-b}, \req{inv-ell-odd-a} and \req{inv-ell-odd-b}
for the inversion of spherical means on elliptical domains in arbitrary spatial
dimension.

\section{Discussion}
\label{sec:discussion}

Many medical imaging  and remote sensing applications aim for recovering a function
from spherical means centered on a set of admissible receiver or detector locations.
In the case that the center set is an infinite hyperplane explicit formulas of the back-projection type for
recovering a function from spherical means are known since the mid 80s (see \cite{And88,Faw85}).
In the case that center set is a spherical or cylindrical surface such type of formulas have been
derived about 20 years later in~\cite{FinHalRak07,FinPatRak04,Kun07a,XuWan05}.
All these geometries have rotational and/or translational invariance and  seem well adapted
to the inversion from spherical means.
It therefore has been believed by many researchers that such exact back-projection type
inversion formulas may only exist for those invariant  geometries.

Very recently,  in \cite{Hal13a,Nat12} explicit exact inversion formulas of the
back-projection type for inverting the  spherical mean
transform with elliptical center sets in dimensions $n=2$ and $n=3$
have been derived
(see \cite{Pal12} for a different formula for ellipsoids  in arbitrary dimension
and \cite{Kun11} for reconstruction formulas for certain polygons and polyhedra).
Moreover, in \cite{Hal13a,Nat12} it has been shown that the same formulas may be
applied when the center set equals the boundary of an arbitrarily shaped smooth
convex  domain $\Om$, in which case  these  formulas recover  the unknown function modulo an explicitly computed integral operator $\K_\Om$.
In the present  paper we generalize these results to the case of arbitrary spatial dimension.
We have further shown, that the operator  $\K_\Om$ vanishes for elliptical domains which yields exact inversion formulas in these cases. However, as can be readily verified by using Equation~\req{kern},
the  operator  $\K_\Om$ does  not vanish for general domains.  Actually, these results give an affirmative negative
answer to the question whether the universal back-projection formula of Xu
and Wang~\cite{XuWan05} (introduced there for the case $n=3$)
is exact for general domains.  This negative result, however, does not imply that a different
back-projection type formula  may  provide exact reconstruction for general domains.
Nevertheless, it  is believed by  the author that such a~\emph{truly universal reconstruction formula}
for the spherical mean transform does not exist.


\def\cprime{$'$} \providecommand{\noopsort}[1]{}

\end{document}